\documentclass[reqno]{amsart}
\usepackage{amsmath,amsfonts,amssymb,amscd,verbatim,multicol}
\usepackage[arrow,matrix,cmtip,curve]{xy}
\usepackage{lscape}
\usepackage{mathrsfs}  
\usepackage{fullpage}
\usepackage{xcolor}
\allowdisplaybreaks

\usepackage[normalem]{ulem}
\numberwithin{equation}{section}

\theoremstyle{definition}
\newtheorem{theorem}{Theorem}[section]
\newtheorem{conjecture}[theorem]{Conjecture}
\newtheorem{corollary}[theorem]{Corollary} 
\newtheorem{definition}[theorem]{Definition} 
\newtheorem{example}[theorem]{Example}
\newtheorem{lemma}[theorem]{Lemma}
\newtheorem{proposition}[theorem]{Proposition}
\newtheorem{question}[theorem]{Question}

\DeclareMathOperator\gr{gr}

\DeclareMathOperator\GL{GL}

\DeclareMathOperator\PAut{PAut}

\renewcommand\int{\mathrm{int}}

\newcommand\inv{^{-1}}
\newcommand\niso{\ncong}
\newcommand\iso{\cong}
\newcommand\tensor{\otimes}

\newcommand\kk{\Bbbk}

\newcommand\cL{\mathcal L}

\newcommand\cR{\mathcal R}

\newcommand\pw{\mathcal W}

\newcommand\NN{\mathbb N}

\newcommand\ZZ{\mathbb Z}

\newcommand\bc{\mathbf c}

\newcommand\bq{\mathbf q}

\newcommand\fg{\mathfrak g}

\newcommand\PC{P_\bc}

\newcommand\grp[1]{\langle #1 \rangle}

\begin{document}

\title{Hopf actions on Poisson algebras}

\author[Alqahtani]{Awn Alqahtani}
\address{Najran University, Department of Mathematics, King Abdulaziz Road, Najran 55461, Kingdom of Saudi Arabia}
\email{odalqahtani@nu.edu.sa}

\author[Gaddis]{Jason Gaddis}
\address{Miami University, Department of Mathematics, Oxford, Ohio 45056, USA} 
\email{gaddisj@miamioh.edu}

\author[Wang]{Xingting Wang}
\address{Department of Mathematics, Louisiana State University, Baton Rouge, Louisiana 70803, USA}
\email{xingtingwang@lsu.edu}

\begin{abstract}
We study finite-dimensional Hopf actions on Poisson algebras and explore the phenomenon of quantum rigidity in this context. Our main focus is on filtered (and especially quadratic) Poisson algebras, including the Weyl Poisson algebra in $2n$ variables and certain Poisson algebras in two variables. In particular, we show that any finite-dimensional Hopf algebra acting inner faithfully on these Poisson algebras must necessarily factor through a group algebra—mirroring well-known rigidity theorems for Weyl algebras in the associative setting. The proofs hinge on lifting the Hopf actions to associated Rees algebras, where we construct suitable noncommutative “quantizations” that allow us to leverage classification results for Hopf actions on quantum (or filtered) algebras. We also discuss how group actions on Poisson algebras extend to universal enveloping algebras, and we give partial classifications of Taft algebra actions on certain low-dimensional Poisson algebras.
\end{abstract}

\maketitle

\section{Introduction}
Quantum rigidity for Poisson algebras states that the automorphism groups of Poisson algebras with nontrivial brackets are, in a sense, smaller than those with trivial brackets. A similar story plays out for Artin--Schelter regular algebras, which are connected to quadratic Poisson algebras through deformation-quantization. See, for example, the work of Ma \cite{cma1}, which computes the automorphism groups for unimodular polynomial Poisson algebras in three variables. This paper contributes to the search for further symmetries through Hopf algebra actions on Poisson algebras (so-called \emph{quantum symmetries}).

However, there is an immediate obstruction to this study. A Poisson algebra is herein defined as a commutative algebra equipped with a Lie bracket, which is a biderivation. By a result of Etingof and Walton, \cite{EW0}, any action of a semisimple Hopf algebra on a commutative domain factors through a group algebra. This suggests that quantum symmetries for Poisson domains only arise from nonsemisimple Hopf actions. On the other hand, the work of Allman classifies actions by a Taft algebra on polynomial rings \cite{All}. Taft algebras are an important class of non-semisimple, pointed Hopf algebras. Their actions on various families of Artin--Schelter regular algebras have also been considered \cite{CG,CGW,GWY}. Our work will focus on actions of \emph{non-semisimple} Hopf algebras on Poisson polynomial algebras.

We are also interested in how rigidity results in the associative setting extends to the Poisson case. For example, Cuadra, Etingof, and Walton have shown that the Weyl algebras over a field of characteristic zero are rigid in the sense that any action of a finite-dimensional Hopf action factors through a group algebra \cite{CEW2} (see also \cite{CWWZ2014,CEW1}). We consider a similar question for \emph{filtered} actions on the Weyl Poisson algebra.

In the study of quantum symmetries, the concept of rigidity can assume several meanings. 
An algebra may be called \emph{rigid} if its invariant ring (by a group or Hopf action) is not isomorphic to the ambient ring. In some settings, such as Artin--Schelter regular algebras, this may be generalized to asking whether the invariant ring maintains certain structural properties. The Weyl algebra is rigid with respect to the action of any finite group \cite{AP}, and the Weyl Poisson algebra is also rigid \cite[Theorem 1.1]{TIK2}. The work in \cite{cma1} and \cite{GVW} shows that various families of unimodular Poisson algebras are rigid in this sense.

In Section \ref{sec.back} we give precise definitions for the main objects of study, including Poisson algebras, gradings and filtrations, and unimodularity. We also introduce Hopf actions and coactions on Poisson algebras. In Section \ref{sec.extension}, we give conditions for an action of a Poisson algebra to extend to the its universal enveloping algebra (Theorem \ref{thm.enveloping}), the Rees algebra of a filtered Poisson algebra (Theorem \ref{thm.rees}), and quantizations of Poisson polynomial algebras (Theorem \ref{thm.quant}). Section \ref{sec.taft} classifies Taft algebra actions on Poisson polynomial algebras (Theorem \ref{thm.taft}).

\subsection*{Acknowledgments}

\section{Background}
\label{sec.back}

Throughout, $\kk$ is an algebraically closed field of characteristic zero, and all algebras are $\kk$-algebras.
A Poisson algebra $A$ is a commutative algebra along with a bracket $\{,\}:A \times A \to A$ such that $(A,\{,\})$ is a Lie algebra and $\{a,-\}:A \to A$ is a derivation for each $a \in A$.

\subsection{Graded Poisson algebras}

Let $\Gamma$ be an abelian monoid. The Poisson algebra $A$ is \emph{$\Gamma$-graded} if there is a vector space decomposition $\bigoplus_{\gamma \in \Gamma} A_{\gamma}$ such that $A_{\gamma_1} \cdot A_{\gamma_2} \subset A_{\gamma_1+\gamma_2}$ and $\{ A_{\gamma_1}, A_{\gamma_2} \} \subset A_{\gamma_1+\gamma_2}$ for all $\gamma_1,\gamma_2 \in \Gamma$. Dropping the Poisson condition gives the usual definition of a $\Gamma$-graded algebra.

\begin{example}\label{ex.quadratic}
Let $P=\kk[u_1,\hdots,u_n]$ be a polynomial Poisson algebra. There is a natural $\NN$-grading on the algebra $P$ obtained by setting $\deg(u_i)=1$. Here $P_k$ denotes the vector space of polynomials of degree $k$ (along with zero). The Poisson algebra $P$ is said to be \emph{quadratic} if it has the usual polynomial grading above and $\{u_i,u_j\} \in P_2$ for all $i,j$.
\end{example}

\subsection{Filtered Poisson algebras}

An \emph{$\NN$-filtration} $F$ on a Poisson algebra $A$ is a collection of subspaces $\{F_i A\}_{i \in \NN}$ satisfying $F_i A \subseteq F_{i+1} A$, $\bigcup_{i \geq 0} F_i A = A$, $(F_i A) \cdot (F_j A) \subseteq F_{i+j}A$, and $\{ F_i A, F_j A\} \subseteq F_{i+j}A$ for all $i,j \in \NN$. Dropping the Poisson condition gives the usual definition of an $\NN$-filtered algebra. Technically, we have defined an \emph{ascending} filtration, and a \emph{descending} filtration may be defined analogously.

\begin{example}
Let $P=\kk[u_1,\hdots,u_n]$ be a Poisson polynomial algebra. Again, set $\deg(u_i)=1$ for all $i$ and let $F_k P$ denote the vector space of polynomials of degree at most $k$. This defines a filtration on the algebra $P$. The Poisson algebra $P$ is \emph{filtered quadratic} if $\{u_i,u_j\} \in F_2 P$ for all $i,j$.  
\end{example}

Given an algebra $A$ with a filtration $F$, the \emph{associated graded ring} is defined as
\[ \gr_F(A) = \bigoplus_{m \geq 0} (F_m A)/(F_{m-1} A)\]
with $F_{-1}=0$. The \emph{Rees algebra} is defined as
\[ R_F(A) = \bigoplus_{n \geq 0} (F_n A)t^n \subseteq A[t].\]
Note that $R_F(A)/(t-1) \iso A$ and $R_F(A)/(t) \iso \gr_F(A)$. We drop the subscript $F$ when it is clear from context.

If $A$ is a Poisson algebra and $F$ is a Poisson filtration, then $\gr(A)$ and $R(A)$ are Poisson algebras with induced brackets. In particular, the bracket on $R(A)$ is induced from the natural bracket on $A[t]$:
\[ \{ at^n,bt^m\}_{A[t]} = \{a,b\}_A t^{n+m}.\]
That $\{ F_n A, F_m A \} \subseteq F_{n+m}A$ ensures that $R(A)$ is closed under the bracket.

\begin{example}\label{ex.pweyl}
Let $\pw = \kk[u_1,v_1,\hdots,u_n,v_n]$ be the $n$th Weyl Poisson algebra with bracket
\[ \{u_i,u_j\} = \{v_i,v_j\} = 0, \qquad \{u_i,v_j\} = \delta_{ij},\]
where $\delta_{ij}$ is the Kronecker-delta.

Note that $\pw$ is filtered quadratic. There is a standard filtration $F$ on $\pw$ where we set $\deg(u_i)=\deg(v_i)=1$ and let $F_n \pw$ be the span of polynomials of degree at most $n$. Under this filtration, $\gr(\pw) = \pw$ with trivial bracket. On the other hand, $R(\pw)=\pw[t]$ with bracket
\[ \{u_i,u_j\} = \{v_i,v_j\} = \{ u_i,t\} = \{v_i,t\} = 0, \qquad \{u_i,v_j\} = \delta_{ij} t^2.\]
\end{example}

\subsection{Unimodular Poisson algebras}

The \emph{modular derivation} $\phi$ of a polynomial Poisson algebra $P=\kk[u_1,\hdots,u_n]$ is given by
\[ \phi(f) = \sum_{j=1}^n \frac{\partial}{\partial u_j} \{f,u_j\}\]
for all $f \in P$. It is easy to check that $\phi$ is a Poisson derivation of $P$, that is, both an algebra derivation and a Lie algebra derivation. We say $P$ is \emph{unimodular} if $\phi=0$. 

\begin{example}
It is well-known that the Weyl Poisson algebra $\pw$ in Example \ref{ex.pweyl} is unimodular. It is straightforward to verify that $R(\pw)$ is again unimodular. Let $\phi$ be the modular derivation. For any $i$,
\begin{align*}
    \phi(u_i) = \sum_{j=1}^n \left( \frac{\partial}{\partial u_j}\{u_i,u_j\} + \frac{\partial}{\partial v_j}\{u_i,v_j\} \right) + \frac{\partial}{\partial t}\{u_i,t\}
        = \frac{\partial}{\partial v_i}\{u_i,v_i\}
        = \frac{\partial}{\partial v_i} t^2 = 0.
\end{align*}
Similarly, $\phi(v_i)=0$ for all $i$. Moreover, it is clear that $\phi(t)=0$.
\end{example}

\subsection{Hopf actions on Poisson algebras}
Throughout, we use (sumless) Sweedler notation for the comultiplication of a Hopf algebra. 

\begin{definition}
Let $H$ be a Hopf algebra. An algebra $A$ is a \emph{(left) $H$-module algebra} if $h \tensor a \to h \cdot a$ makes $A$ a (left) $H$-module, that is $h \cdot 1_A = \epsilon(h) 1_A$ and $h \cdot (ab) = (h_1 \cdot a)(h_2 \cdot a)$ for all $h \in H$ and all $a,b \in A$.

A Poisson algebra $A$ is a \emph{(left) $H$-module Poisson algebra} if $A$ is a (left) $H$-module algebra and $h \cdot \{a,b\} = \{ h_1 \cdot a, h_2 \cdot b \}$ for all $h \in H$ and all $a,b \in A$.
\end{definition}

\begin{example}
(1) Let $G$ be a subgroup of Poisson automorphisms of a Poisson algebra $A$. Then for $g \in G$, $g \cdot \{a,b\} = \{ g \cdot a, g \cdot b\}$. This agrees with the definition of a group action on a Poisson algebra \cite{AF} and so $A$ is a $\kk G$-module Poisson algebra.

(2) Let $\fg$ be the Lie algebra of Poisson derivations of a Poisson algebra $A$. For $\delta \in \fg$,
\[ \delta(\{a,b\}) = \{ \delta(a), b \} + \{ a, \delta(b) \},\]
and so $A$ is a $U(\fg)$-module Poisson algebra.
\end{example}

Suppose $H$ is a Hopf algebra and $A$ is an $H$-module algebra.
If $I$ is a Hopf ideal of $H$ such that $IA=0$, then $H/I$ acts naturally on $A$.
If $IA=0$ implies $I=0$ for all Hopf ideals $I$ of $H$, then we say $H$
acts \emph{inner faithfully} on $A$.  That is, the action of $H$ on $A$ does
not factor through the action of a ``smaller" Hopf algebra.

If $H$ is a cosemisimple Hopf algebra that acts inner faithfully on a commutative domain $A$, then $H$ is a finite group algebra \cite[Theorem 1.3]{EW0}. Note that cosemisimple and semisimple are equivalent over a field of characteristic zero, as we have assumed. Hence, we concentrate on the case where $H$ is non-semisimple.

Suppose $H$ acts on $A$. Define the fixed ring
\[ A^H = \{ a \in A : h \cdot a = \epsilon(h)a \text{ for all } h \in H\}.\]

\begin{lemma}
\label{lem.fixed}
If $A$ is a Poisson algebra and $H$ acts on $A$, then $A^H$ is a Poisson subalgebra of $A$.
\end{lemma}
\begin{proof}
It is well-known that $A^H$ is a subalgebra of $A$.
We need only show that $A^H$ is closed under the bracket.
Let $a,b \in A^H$ and $h \in H$. Then
\[ h \cdot \{a,b\} = \{ h_1 \cdot a, h_2 \cdot b \}
	= \{ \epsilon(h_1) a, \epsilon(h_2) b\}
	= \epsilon(h_1)\epsilon(h_2) \{a,b\}
	= \epsilon(h) \{a,b\}. \qedhere\]
\end{proof}

\subsection{Comodule Poisson algebras}

We now define the dual notion to Hopf module Poisson algebras above. Given Poisson algebras $(A,\{,\})_A$ and $(B,\{,\})_B$, there is a natural Poisson structure on $A \tensor B$ given by
\[ \{ a_1 \tensor b_1, a_2 \tensor b_2 \} = \{ a_1,a_2\}_A \tensor b_1b_2 + a_1a_2 \tensor \{b_1,b_2\}_B.\]
If $H$ is a Hopf algebra, then we make $A \tensor H$ into a Poisson algebra by assigning 
\[ \{g,h\} = [g,h]=gh-hg \]
for all $g,h \in H$. Similarly for $H \tensor A$.

\begin{definition}
Let $H$ be a Hopf algebra.
An algebra $A$ is a \emph{right $H$-comodule algebra} with structure map $\rho:A \to A \tensor H$ that makes $A$ into a (right) $H$-comodule, $\rho(1_A) = 1_A \tensor 1_H$, and
\[ \rho(ab) = \sum (a_0 b_0) \tensor (a_1 b_1) \]
for all $a,b \in A$ where we write $\rho(a)=a_0 \tensor a_1$ and $\rho(b)=b_0 \tensor b_1$.

A Poisson algebra $A$ is a \emph{right $H$-comodule Poisson algebra} with structure map $\rho:A \to A \tensor H$ that makes $A$ a (right) $H$-comodule algebra and
\[ \rho(\{a,b\}) = \sum \{ a_0, b_0 \} \tensor a_1b_1 \]
for all $a,b \in A$.

One defines a left $H$-comodule (Poisson) algebra similarly.
\end{definition}

\begin{example}\label{ex.coaction}
Let $c=c_{12} \in \kk$ and let $\PC=\kk[u_1,u_2]$, so $\{u_1,u_2\}=cu_1u_2$. Let $H=\kk[t^{\pm 1}]$, which is a Hopf algebra with structure maps $\Delta(t^n) = t^n \tensor t^n$, $\epsilon(t^n) = 1$, and $S(t^n) = t^{-n}$ for all $n \in \ZZ$. Define a right coaction $\rho:\PC \to \PC \tensor H$ by $\rho(u_1) = u_1 \tensor t$ and $\rho(u_2) = u_2 \tensor t\inv$. It is easy to verify that $\rho$ is a Poisson algebra homomorphism:
\begin{align*}
\rho(u_1)\rho(u_2)-\rho(u_2)\rho(u_1) 
    &= (u_1 \tensor t)(u_2 \tensor t\inv) - (u_2\tensor t\inv)(u_1\tensor t)
    = 0 
    = \rho(u_1u_2-u_2u_1) \\
\{ \rho(u_1),\rho(u_2)\} 
    &= \{ u_1 \tensor t, u_2 \tensor t\inv\} 
    = c u_1u_2 \tensor 1 
    = c(u_1 \tensor t)(u_2 \tensor t\inv) 
    = \rho(\{u_1,u_2\}).
\end{align*}
Thus, $\PC$ is a right $H$-comodule Poisson algebra.
\end{example}

The following proposition is well-known in the associative algebra setting, see e.g., \cite[Proposition 6.2.4]{DNR}.

\begin{proposition}
Let $H$ be a Hopf algebra and $A$ a Poisson algebra.
\begin{enumerate}
\item If $A$ is a right $H$-comodule Poisson algebra, then $A$ is a left $H^*$-module Poisson algebra.
\item If $A$ is a left $H^*$-module Poisson algebra and $H$ is finite dimensional, then $A$ is a right $H$-comodule Poisson algebra.
\end{enumerate}
\end{proposition}
\begin{proof}
(1) Let $\rho:A \to A \tensor H$ be the structure map 
using conventions as before. For $f \in H^*$ define $f \cdot a = \sum f(a_1)a_0$. 
It is well-known that this makes $A$ into a left $H$-module algebra. It remains only to verify that it respects the Poisson bracket on $A$.  That is,
\[ 
f \cdot \{a,b\} 
	= \sum f(a_1b_1) \{ a_0, b_0 \}
	= \sum f_1(a_1) f_2(b_1) \{ a_0, b_0 \}
	= \sum \{ f_1(a_1) a_0, f_2(b_1)  b_0 \}
	= \sum \{ f_1 \cdot a, f_2 \cdot b \}.
\]

(2) Set $n=\dim_\kk(H)$ and choose dual bases $\{e_1,\hdots,e_n\} \subset H$ and $\{e_1^*,\hdots,e_n^*\} \subset H^*$. The map $\rho: A \to A \tensor H$ given by
$\rho(a) = \sum_{i=1}^n e_i^* \cdot a \tensor e_i$ makes $A$ into a right $H$-comodule algebra. 
Then for $f \in H^*$,
\begin{align*}
(I \tensor f)(\rho(\{a,b\})
	&= \sum_{i=1}^n e_i^* \cdot \{a,b\} \tensor f(e_i)
	= \sum_{i=1}^n \left( e_i^* f(e_i) \right) \cdot \{a,b\} \tensor 1
	= f \cdot \{a,b\} \tensor 1 \\
	&= \sum \{ f_1 \cdot a, f_2 \cdot b \} \tensor 1
	= \sum_{i,j = 1}^n \{ \left( e_i^* f_1(e_i) \right) \cdot a, \left( e_j^* f_2(e_j) \right) \cdot b \} \tensor 1 \\
	&= \sum_{i,j = 1}^n \{ e_i^* \cdot a, e_j^* \cdot b \} \tensor f_1(e_i)f_2(e_j)
	= \sum_{i,j = 1}^n \{ e_i^* \cdot a, e_j^* \cdot b \} \tensor f(e_ie_j) \\
	&= (I \tensor f)\left( \sum_{i,j = 1}^n \{ e_i^* \cdot a, e_j^* \cdot b \} \tensor e_ie_j \right)
	= (I \tensor f)( \{ \rho(a), \rho(b) \} ).
\end{align*}
That is, $\rho$ is a Poisson homomorphism.
\end{proof}

\section{Extending actions}
\label{sec.extension}

In this section, we consider how actions on Poisson algebras extend to universal enveloping algebras, Rees algebras, and quantizations. Part of the motivation for studying these is to demonstrate that the Weyl Poisson algebra is rigid with respect to certain Hopf algebra actions.

\subsection{Universal enveloping algebras}

Let $A$ be a Poisson algebra. The universal enveloping algebra $U(A)$ of $A$ is the associative algebra that is universal with respect to the existence of an algebra embedding $\mu:A \to U(A)$ and a Lie algebra homomorphism $\nu:(A,\{,\}) \to (U(A),[,])$ \cite{OH3}. In the literature, the maps $\mu$ and $\nu$ are often referred to as $m$ and $h$, respectively. However, we wish to avoid conflict with the usage of $h$ as a generic element of a Hopf algebra $H$.

In terms of generators and relations, $U(A)$ is the associative algebra over $\kk$ with identity $1$ generated by $\mu_a = \mu(a)$ and $\nu_a = \nu(a)$ for all $a \in A$. The relations on these generators are:
\begin{align*}
\mu_{ab}&= \mu_a\mu_b, &
\mu_{\{a,b\}}&=\nu_a\mu_b-\mu_b\nu_a, & 
\mu_1 &= 1, \\
\nu_{ab}&=\mu_b\nu_a + \mu_a\nu_b, &
\nu_{\{a,b\}}&=\nu_a\nu_b-\nu_b\nu_a.
\end{align*}
for any $a,b\in A$ \cite{UU}.

\begin{example}\label{ex.skewsym}
Let $\bc=(c_{ij})\in M_n(\kk)$ be a skew-symmetric matrix and $\PC=\kk[u_1,\ldots,u_n]$ be the quadratic Poisson algebra with Poisson bracket given by 
\[ \{u_i,u_j\} = c_{ij}u_iu_j \]
for all $1\leq i,j\leq n$. Then $U(\PC)$ is generated over $\kk$ by $\mu_i=\mu_{u_i}$ and $\nu_i = \nu_{u_i}$ for $1\leq i\leq n$ with relations
\[
[\mu_i,\mu_j] = [\mu_i,\nu_i] = 0, \qquad
[\nu_i,\mu_j] = c_{ij} \mu_i \mu_j, \qquad
[\nu_i,\nu_j] = c_{ij} (\mu_i\nu_j + \mu_j\nu_i),
\]
for all $1\leq i,j\leq n$, $i \neq j$.
\end{example}

Let $G$ be a group acting on the Poisson algebra $A$. In \cite{GVW} it was shown that this action extends to an action on $U(A)$ via $g \cdot \mu_a = \mu_{g \cdot a}$ and $g \cdot \nu_a = \nu_{g \cdot a}$ for all $g \in G$, $a \in A$. See also the work of Ma related to invariants of such actions \cite{cma1}. We aim to generalize this to certain Hopf algebra actions.

\begin{theorem}\label{thm.enveloping}
Let $H$ be a Hopf algebra and $A$ an $H$-module Poisson algebra.
Then $U(A)$ is an $H$-module algebra via the action
\begin{align}\label{eq.hopf_ext}
	h \cdot \mu_a = \mu_{h \cdot a}, \qquad 
    h \cdot \nu_a = \nu_{h \cdot a},
\end{align}
for all $h \in H$ and $a \in A$ if and only if 
\begin{align}
\label{eq.hopf_ext_cond1}
\mu_{h_1 \cdot b}\nu_{h_2 \cdot a} &= \mu_{h_2 \cdot b} \nu_{h_1 \cdot a}, \\
\label{eq.hopf_ext_cond2}
\nu_{h_2 \cdot b}\nu_{h_1 \cdot a} &= \nu_{h_1 \cdot b}\nu_{ h_2 \cdot a},
\end{align}
for all $h \in H$ and all $a,b \in A$.
\end{theorem}
\begin{proof}
For $h \in H$, we extend the action of $h$ on $A$ by \eqref{eq.hopf_ext}.
We verify that $U(A)$ is indeed an $H$-module algebra by showing that $h$ respects the relations of $U(A)$.

Let $a,b \in A$. First we have
\[
h \cdot \mu_{ab}
	= \mu_{h \cdot (ab)} 
 = \mu_{(h_1 \cdot a)(h_2 \cdot b)} 
    =  \mu_{h_1 \cdot a} \mu_{h_2 \cdot b}
    = (h_1 \cdot \mu_a) (h_2 \cdot \mu_b)
    = h \cdot (\mu_a \mu_b).
\]
Next we have
\begin{align*}
h \cdot \left(\mu_{\{a,b\}} - ( \nu_a\mu_b-\mu_b\nu_a ) \right)
    &= \mu_{\{h_1 \cdot a, h_2 \cdot b\}} - (\nu_{h_1 \cdot a}\mu_{h_2 \cdot b} - \mu_{h_1 \cdot b}\nu_{h_2\cdot a}) \\
    &= (\nu_{h_1 \cdot a} \mu_{h_2 \cdot b} - \mu_{h_2 \cdot b} \nu_{h_1 \cdot a}) - (\nu_{h_1 \cdot a}\mu_{h_2 \cdot b} - \mu_{h_1 \cdot b}\nu_{h_2\cdot a}) \\
    &= \mu_{h_1 \cdot b}\nu_{h_2 \cdot a} - \mu_{h_2 \cdot b} \nu_{h_1 \cdot a}.
\end{align*}
Hence, if $H$ acts on $U(A)$, then \eqref{eq.hopf_ext_cond1} holds.

Similarly,
\begin{align*}
h \cdot \left(\nu_{ab} - (\mu_b \nu_a + \mu_a \nu_b)\right)
    &= \nu_{(h_1 \cdot a)(h_2 \cdot y)} - (\mu_{h_1 \cdot b} \nu_{h_2 \cdot a} + \mu_{h_1 \cdot a} \nu_{h_2 \cdot b}) \\
    &= (\mu_{h_2 \cdot b} \nu_{h_1 \cdot a} + \mu_{h_1 \cdot a} \nu_{h_2 \cdot b}) - (\mu_{h_1 \cdot b} \nu_{h_2 \cdot a} + \mu_{h_1 \cdot a} \nu_{h_2 \cdot b}) \\
    &= \mu_{h_2 \cdot b} \nu_{h_1 \cdot a}  - \mu_{h_1 \cdot b} \nu_{h_2 \cdot a},
\end{align*}
which again implies \eqref{eq.hopf_ext_cond1}.

Finally, we have
\begin{align*}
h \cdot \left(\nu_{\{a,b\}} - (\nu_a \nu_b - \nu_b \nu_a) \right) 
    &= \nu_{ \{h_1 \cdot a, h_2 \cdot b\} } - (\nu_{h_1 \cdot a} \nu_{h_2 \cdot b} - \nu_{h_1 \cdot b}\nu_{h_2 \cdot a}) \\
    &= (\nu_{h_1 \cdot a} \nu_{h_2 \cdot b} - \nu_{h_2 \cdot b}\nu_{h_1 \cdot a}) - (\nu_{h_1 \cdot a} \nu_{h_2 \cdot b} - \nu_{h_1 \cdot b}\nu_{h_2 \cdot a}) \\
    &= \nu_{h_1 \cdot b}\nu_{h_2 \cdot a} - \nu_{h_2 \cdot b}\nu_{h_1 \cdot a}.
\end{align*}
Hence, if $H$ acts on $U(A)$, then \eqref{eq.hopf_ext_cond2} holds.

Conversely, it is clear from the above computations that if the two conditions hold, then the action of $H$ on $U(A)$ is well-defined.
\end{proof}

\begin{corollary}\label{cor.cocomm}
Keep the setup of Theorem \ref{thm.enveloping},
so $A$ is a Poisson algebra.
If $H$ is cocommutative, then \eqref{eq.hopf_ext_cond1} and \eqref{eq.hopf_ext_cond2} hold. Hence, $U(A)$ is an $H$-module algebra.

In particular, let $\cL$ be a Lie algebra of Poisson derivations of $A$. Then $U(\cL)$ is an infinite-dimensional Hopf algebra which acts naturally on $A$. Thus, $U(\cL)$ acts on $U(A)$.
\end{corollary}

\begin{corollary}\label{cor.weyl1}
Let $\pw$ be the $n$th Weyl Poisson algebra and let $H$ be a finite-dimensional Hopf algebra acting on $A$. If the action satisfies \eqref{eq.hopf_ext_cond1} and \eqref{eq.hopf_ext_cond2}, then the action extends to an action on $U(\pw)$ via \eqref{eq.hopf_ext}. But then $U(\pw)$ is the $2n$th Weyl algebra and so by \cite[Theorem 1.1]{CEW2}, the action factors through the action of a group algebra.
\end{corollary}

\subsection{Rees algebras}

Let $A$ be a (Poisson) algebra, $F$ a (Poisson) filtration on $A$, and $H$ a Hopf algebra that acts on $A$. We say the $H$-action \emph{respects the filtration} if each $F_n A$ is an $H$-submodule.

\begin{theorem}\label{thm.rees}
Let $A$ be a Poisson algebra with Poisson filtration $F$. Let $H$ be a finite-dimensional Hopf algebra that acts on $A$ and respects the filtration $F$. Then there is an induced action of $H$ on $R(A)$ such that each homogeneous component of $R(A)$ is an $H$-module.
\end{theorem}
\begin{proof}
We extend the action of $H$ on $A$ to an action of $H$ on the Rees algebra $R(A)$ by setting $h \cdot t = \epsilon(h)t$. An inductive argument shows that $h(t^k) = \epsilon(h)t^k$ for all $k \geq 1$. That this defines an action on the algebra $A$ follows from \cite[Lemma 4.2]{CWWZ2014}. It remains only to check the Poisson structure. Let $a \in A$ and $h \in H$. Then
\[
h \cdot (\{t,a\}) = \{ h_1 \cdot t, h_2 \cdot a\}
    = \{ \epsilon(h_1)t, h_2 \cdot a\}
    = \epsilon(h_1) \{ t, h_2 \cdot a\}
    = 0.
\]
Now, let $a \in F_n A$ and $b \in F_m A$. Then
\begin{align*}
\{ h(at^n),h(bt^m) \}
    &= \{ h_1(a)h_2(t^n), h_3(b)h_4(t^m) \} \\
    &= \epsilon(h_2)\epsilon(h_4) \{ h_1(a),h_3(b) \} t^{n+m} \\
    &= \epsilon(h_2) h_1(\{ a,b \} t^{n+m} )\\
    &= h_1(\{a,b\}) h_2(t^{n+m}) \\
    &= h \cdot \left(\{a,b\} t^{n+m} \right) \\
    &= h \cdot \{ at^n,bt^m \}.
\end{align*}
This gives the result.
\end{proof}

\subsection{Quantizations}
In this section, we establish the rigidity of quadratic Poisson algebras relative to finite-dimensional Hopf actions by demonstrating explicitly how one may lift actions to quantizations.

Let $P=\kk[x_1,\ldots,x_n]$ be a quadratic Poisson algebra. Suppose $H$ is a finite-dimensional Hopf algebra acting on $P$ while preserving its gradings. So, the generating space $P_1=\kk x_1\oplus \cdots \oplus \kk x_n$ is an $H$-module. For any scalar $\lambda$, we set 
\begin{align}\label{KerT}
\cR_\lambda~:=~{\rm Ker}\left(P_1\otimes P_1\hookrightarrow P\otimes P\xrightarrow{\sigma_\lambda} P\right) 
\end{align}
where $\sigma_\lambda=\mu+\lambda\{-,-\}$ such that $\mu: P\otimes P\to P$ is the multiplication on $P$. We denote the corresponding  associative graded algebra
\[P_\lambda~:=~\kk\langle P_1\rangle/(\cR_\lambda).\]

\begin{theorem}\label{thm.quant}
Let $P=\kk[x_1,\ldots,x_n]$ be a quadratic Poisson algebra and let $H$ be a finite-dimensional Hopf algebra acting on $P$ while preserving degree. For any $\lambda\in \kk$, there is a corresponding graded $H$-action on the associative graded algebra $P_\lambda$. Moreover, if there is some $\lambda\in \kk$ such that any graded finite-dimensional Hopf action on $P_\lambda$ must factor through a group algebra, then the $H$-action on $P$ also factors through a group algebra.  
\end{theorem}
\begin{proof}
Since $\sigma_\lambda$ is $H$-linear, then $\cR_\lambda$ is an $H$-submodule of $P_1\otimes P_1$. As a consequence, the $H$-module structure on $P_1$ induces a graded $H$-action on the associative algebra $P_\lambda=\kk\langle P_1\rangle/(\cR_\lambda)$. Moreover, it is clear that $P_0\cong P$ and $H$ acts inner-faithfully on $P_\lambda$ if it does on $P$. 

Now, suppose there is some $\lambda\in \kk$ such that any graded finite-dimensional Hopf action on $P_\lambda$ must factor through a group algebra. Then there is a Hopf ideal $I$ of $H$ such that $IP_1=0$ and $H/I$ is a group algebra. Certainly, $I(P_1\otimes P_1)=0$ and $\mu,\{-,-\}: P_1\otimes P_1\to P_1$ are $H/I$-module maps. Therefore, the $H$-action on the polynomial Poisson algebra $P$ factors through $H/I$.
\end{proof}

We now give several applications of Theorem \ref{thm.quant}. Let $\bq=(q_{ij})$ and we can view $\bq$ as a point of the algebraic torus $(\kk^\times)^{n(n-1)/2}$ with coordinates $q_{ij}$, $i<j$. Let $\langle \bq\rangle$ be the subgroup in $(\kk^\times)^{n(n-1)/2}$ generated by $\bq$, and let $G(\bq)$ be its Zariski closure, where $G^0(\bq)$ is the connected component of the identity in $G_\bq$.

\begin{corollary}\label{cor.Gskew}
Let $\bc \in M_n(\kk)$ be skew-symmetric and let $\PC$ be the corresponding Poisson algebra defined in Example \ref{ex.skewsym}. Assume there is a generic number $\lambda$ with $\bq=(q_{ij})$ where $q_{ij}=\frac{1+\lambda c_{ij}}{1-\lambda c_{ij}}$ such that the order of $G(\bq)/G^0(\bq)$ is coprime to $d!$, and the bicharacter $\bq$ is nondegenerate. If $H$ is a finite-dimensional Hopf algebra of dimension $d$ acting on $\PC$ while preserving the grading, then the action of $H$ on $\PC$ factors through a group algebra.
\end{corollary}
\begin{proof}
Because $\lambda$ is generic, $P_\lambda$ is a graded algebra, and we have a surjection $\kk\langle v_1,\ldots,v_n\rangle\twoheadrightarrow P_\lambda$ via $v_i\mapsto u_i$. Since
\[
\sigma_\lambda(u_i\otimes u_j)=\mu(u_i\otimes u_j)+\lambda\{u_i,u_j\}=(1+\lambda\,c_{ij})u_iu_j=q_{ij}(1-\lambda\, c_{ij})u_iu_j=q_{ij}\sigma_\lambda(u_j\otimes u_i),
\]
we have $\{v_i\otimes v_j-q_{ij}v_j\otimes v_i \mid 1\le i,j\le n\}\subseteq \cR_\lambda$. 
Then $S_\bq=\kk\langle v_1,\ldots,v_n\rangle/(v_iv_j-q_{ij}v_jv_i)$ is a skew polynomial ring, which has the same Hilbert series as the polynomial ring $P$. In particular, $\cR_\lambda=(v_i\otimes v_j-q_{ij}v_j\otimes v_i \mid 1\le i,j\le n)$ by a dimension counting argument. Consequently, $S_\bq \cong P_\lambda$ as associative algebras. The result now follows from the corresponding result for finite-dimensional Hopf actions on skew polynomial rings \cite[Theorem  1.9]{EW2016}.
\end{proof}

The next result establishes that the $n$th Weyl Poisson algebra $\pw$ is rigid with respect to linear actions by finite-dimensional Hopf algebras. We conjecture that the result holds without the linear hypothesis. See Conjecture \ref{conj.pweyl}.

\begin{corollary}\label{cor.weyl}
Let $H$ be a finite-dimensional Hopf algebra acting on the $n$th Weyl Poisson algebra $\pw$ while preserving its standard filtration. Then, the action of $H$ factors through a group algebra. 
\end{corollary}
\begin{proof}
Recall that $R(\pw)=\kk[u_1,v_1,\ldots,u_n,v_n,t]$ is the Rees algebra of $\pw$, as discussed in Example \ref{ex.pweyl}. Define the quantization of $R(\pw)$ as
\[ B = \kk\langle x_1,y_1,\ldots,x_n,y_n,z\rangle/(x_iy_j-y_jx_i,x_ix_j-x_jx_i-\delta_{ij}z^2,y_iy_j-y_jy_i-\delta_{ij}z^2,zx_i-x_iz,zy_i-y_iz).\]
The mappings $x_i\mapsto u_i$, $y_i\mapsto v_i$, $z\mapsto t$ define a surjection $B \twoheadrightarrow R(\pw)$. Since $B$ has the same Hilbert series as the polynomial algebra $R(\pw)$, then 
\[\cR_{1/2}=(u_iv_j-v_ju_i,u_iu_j-u_ju_i-\delta_{ij}t^2,v_iv_j-v_jv_i-\delta_{ij}t^2,tu_i-u_it,tv_i-v_it),\] 
and hence $B\cong \kk\langle u_i,v_i,t\rangle/(\cR_{1/2})$ as associative algebras.

By Theorem \ref{thm.rees}, the action of $H$ on $\pw$ extends to an action of $H$ on the Rees algebra $R(\pw)$ by requiring that $H$ acts on $t$ trivially. As discussed above, $H$ acts on $B$, where the action of $H$ on $z$ is again trivial. Hence, $H$ acts on the quotient algebra $B/(t-1)$, which is isomorphic to the $n$th Weyl algebra. Our result now follows from \cite[Theorem 1.1]{CEW2}.
\end{proof}

Our final rigidity result concerns filtered Poisson algebras on a polynomial ring in two variables.

\begin{corollary}\label{cor.filtered}
Let $P=\kk[u_1,u_2]$ be a filtered quadratic Poisson algebra with a nontrivial Poisson bracket. Suppose $H$ is a finite-dimensional Hopf algebra that acts on $P$ while preserving the filtration. Then, the action of $H$ on $P$ factors through a group algebra.  
\end{corollary}
\begin{proof}
Since $P$ is a filtered quadratic Poisson algebra with nontrivial bracket, then $\{u_1,u_2\} = f$ with $\deg(f) \leq 2$. By \cite[Theorem 3.1]{GWY}, up to a change of variable, $f$ is one of the following polynomials:
\[ \quad 1, \quad u_1, \quad u_1^2, \quad u_1^2+1, \quad qu_1u_2, \quad qu_1u_2 + 1.\]
with $q \in \kk^\times$. If $f=1$, then $P$ is the first Weyl Poisson algebra, and so the result in this case follows from Corollary \ref{cor.weyl}.

Let $R(P)$ be the Rees algebra of $P$ with respect to the given (standard) filtration. Let $\widetilde{f}$ be the homogenization of $f$ with respect to $t$, so that $R(P)=\kk[u_1,u_2,t]$ with bracket
\[ \{u_1,u_2\} = \widetilde{f}, \qquad \{t,u_1\}=\{t,u_2\}= 0.\]
By Theorem \ref{thm.rees}, the action of $H$ on $P$ lifts to an action on $R(P)$ by letting $H$ act trivially on $t$. 

Define the quantization of $R(P)$ as
\[ B = \kk\langle x_1,x_2,z \rangle/(x_1x_2-x_2x_1-\widetilde{f}(x_1,x_2),x_1t-tx_1,x_2t-tx_2).\]
The mappings $x_1 \mapsto u_1$, $x_2\mapsto u_2$, $z\mapsto t$ define a surjection $B \twoheadrightarrow R(P)$. Since $B$ has the same Hilbert series as the polynomial algebra $R(P)$, then 
\[\cR_{1/2}=(u_1u_2-u_2u_1-\widetilde{f}(u_1,u_2), u_1t-tu_1, u_2t-tu_2),\] 
and hence $B\cong \kk\langle x,y,t \rangle/(\cR_{1/2})$ as associative algebras. Since $H$ acts trivially on $t$, then $H$ acts on the quotient $B/(t-1)$, which is a filtered Artin--Schelter regular algebra. Since the $H$-action preserves the filtration on $P$, then the action factors through a group algebra by \cite[Theorem 0.1]{CWWZ2014}.
\end{proof}

\section{Taft algebra actions}
\label{sec.taft}

Taft algebras play a crucial role in the study of non-semisimple Hopf algebras. They are \emph{pointed} Hopf algebras, that is, every simple comodule is one-dimensional. All pointed Hopf algebras of dimension $p$ or $p^2$, $p$ a prime, are either a group algebra or a Taft algebra. Actions of Taft algebras on Artin--Schelter regular algebras have been studied in \cite{CG,CGW,GWY}.
Here we consider linear actions on Poisson polynomial algebras.

\begin{definition}\label{defn.taft}
Let $n>1$ be an integer and let 
$\lambda \in \kk^\times$ be a primitive $n$th root of unity.
The \emph{Taft algebra} $T_n(\lambda)$ is the $\kk$-algebra
\[ T_n(\lambda) = \kk\langle g,x \mid gx-\lambda xg, g^n-1, x^n \rangle.\]
The Taft algebra is a Hopf algebra with counit $\epsilon$, comultiplication $\Delta$, and antipode $S$ given by
\begin{align*}
    \epsilon(g) &= 1    &   \Delta(g) &= g \tensor g & S(g) &= g\inv \\
    \epsilon(x) &= 0    &   \Delta(x) &= g \tensor x + x\tensor 1 & S(x) &= -g\inv x.
\end{align*}
That is, $g$ is grouplike and $x$ is $(1,g)$-primitive.
\end{definition}

Let $P=\kk[u_1,\hdots,u_m]$ and $T=T_n(\lambda)$. Assume that $P$ is a $T$-module algebra and that the action of $T$ on $P$ is graded. According to Allman \cite{All}, we may diagonalize so that, up to a change of variable, the action of $T$ on $P$ is given by 
\begin{align}\label{eq.taft}
g(u_i) = \begin{cases}
    \lambda\inv u_1 & \text{if $i=1$} \\
    u_i & \text{if $1 < i \leq m$}
\end{cases} \qquad
x(u_i) = \begin{cases}
    u_2 & \text{if $i=1$} \\
    0 & \text{if $1 < i \leq m$}.
\end{cases}
\end{align}
The invariants under this action are
\begin{align}\label{eq.Taft_inv}
P^T = P^{\grp{g}} \cap P^{\grp{x}} = \kk[u_1^n,u_2,\hdots,u_m].
\end{align}

We will shortly show that it is possible to classify linear Taft actions on polynomial Poisson algebras. First, we establish that all such actions extend to the enveloping algebra.

\begin{lemma}\label{lem.taft1}
Let $P=\kk[u_1,\hdots,u_m]$ be a Poisson algebra. Suppose $T=T_n(\lambda)$ acts linearly on $P$ and the action is given as in \eqref{eq.taft}. Then $\{u_1,u_2\}=0$. Consequently, if $m=2$, then the Poisson bracket is trivial.
\end{lemma}
\begin{proof}
Since $\{u_1,u_1\}=0$, then
\[ 0 = x \cdot \{u_1,u_1\} = \{g\cdot u_1,x \cdot u_1 \} + \{x \cdot u_1,u_1\} = \{ \lambda\inv u_1,u_2\} + \{u_2,u_1\} = (\lambda\inv-1)\{u_1,u_2\}.\]
By definition of the Taft algebra, $\lambda \neq 1$. 
The result follows.
\end{proof}

We now use the above result to classify actions on Poisson polynomial algebras.

\begin{lemma}\label{lem.taft2}
Let $P=\kk[u_1,\hdots,u_m]$, $m>2$, be a Poisson algebra. Suppose $T=T_n(\lambda)$ acts linearly on $P$ and the action is given as in \eqref{eq.taft}. Write $\{u_i,u_j\} = f_{ij}$ for $i>j>1$ and $\{u_i,u_1\}=h_i$ for $i \geq 3$. Then $x(f_{ij})=0$ for all $i,j>1$ and for all $i>2$, $g(h_i)=\lambda\inv h_i$ and $x(h_i)=f_{i2}$.
\end{lemma}
\begin{proof}
The results are immediate from the following calculations:
\begin{align*}
    x(f_{ij}) &= x(\{u_i,u_j\}) = 0, \\
    g(h_i) &= g(\{u_i,u_1\}) = \{g(u_i),g(u_1)\} = \{u_i,\lambda\inv u_1\} = \lambda\inv h_i, \\
    x(h_i) &= x(\{u_i,u_1\}) = \{g(u_i),x(u_1)\} + \{x(u_i),u_1\} = \{u_i,u_2\} = f_{i2}. \qedhere
\end{align*}
\end{proof}

\begin{theorem}\label{thm.taft}
Let $P=\kk[u_1,u_2,u_3]$ be a Poisson algebra. Suppose $T=T_n(\lambda)$ acts linearly on $P$ and the action is given as in \eqref{eq.taft}.
\begin{enumerate}
    \item If $\{u_2,u_3\} \in \kk$, then the bracket on $P$ is trivial.
    \item If $\{u_2,u_3\} \in P_1$, then either the bracket on $P$ is trivial or, up to rescaling, the bracket on $P$ is given by
\begin{align}\label{eq.taft_linear}
    \{ u_1,u_2\} = 0, \qquad 
    \{u_2,u_3\} = u_2, \qquad 
    \{u_1,u_3\} = u_1.
\end{align}
    \item If $\{u_2,u_3\} \in P_2$, then up to a linear change of variable, the bracket on $P$ is given by
\begin{align}\label{eq.quad_taft}
\{ u_1,u_2\} = 0, \qquad
\{ u_2,u_3\} = cu_2u_3, \qquad
\{ u_1,u_3\} = cu_1u_3,
\end{align}
where $c \in \kk$.
\end{enumerate}
\end{theorem}
\begin{proof}
(1) Write $\{u_3,u_2\}=f \in \kk$. If $f \neq 0$, then $f \notin \mathrm{im} x$. Now write $\{u_1,u_3\} = h$. Then $x(h) = 0$ so $h \in \ker x = \kk[u_1^n,u_2,u_3]$. But then $g(h)=h$, so $h=0$.

(2) Write $\{u_3,u_2\}=f$ and $\{u_3,u_1\}=h$ where $f=au_1+bu_2+cu_3$ for $a,b,c \in \kk$. Since $x(f)=0$, then $a=0$. Moreover, since $f \in \mathrm{im} x$, then $c=0$. Since $f=x(h)$, then $h=bu_1$. Note that $g(h)=\lambda\inv h$. Rescaling now gives the result.

(3) Suppose $\lambda \neq -1$. Since $x(f)=0$, then $f=bu_2^2 + cu_2u_3 + du_3^2$. As $f=x(h)$, then it follows that $f=bu_2^2 + cu_2u_3$ and $h=bu_1u_2 + cu_1u_3$. It is clear that $g(h)=\lambda\inv h$.

If $\lambda=-1$, then $f=au_1^2+bu_2^2 + cu_2u_3 + du_3^2$. However, since $u_1^2 \notin \mathrm{im} x$, then we may assume $a=0$ and reduce to the above computation.

It is left only to check the Jacobi identity in order to verify that this is a valid Poisson structure. We do this directly:
\begin{align*}
\{ u_1, \{u_2,u_3\} \} &+ \{ u_2, \{u_3,u_1\} \} + \{ u_3, \{u_1,u_2\} \} \\
    &= -\{ u_1, bu_2^2 + cu_2u_3\} + \{ u_2, bu_1u_2 + cu_1u_3\} \\
    &= -c u_2 \{ u_1,u_3\} + cu_1 \{u_2,u_3\} \\
    &= c \left( u_2 (bu_1u_2 + cu_1u_3) - cu_1 (bu_2^2 + cu_2u_3) \right) = 0.
\end{align*}

Finally, we make the substitution $u_3' = -(bu_2+cu_3)$ 
and this does not affect the Taft action.
That is, $g(u_3')=u_3'$ and $x(u_3')=0$. Then the bracket 
on $\kk[u_1,u_2,u_3]$ is given by $\{u_1,u_2\}=0$ and
\begin{align*}
\{u_2,u_3'\}
    &= \{ bu_2+cu_3, u_2\} 
    = \{ cu_3,u_2\} = c(bu_2^2 + cu_2u_3)
    = cu_2 (bu_2 + cu_3) = cu_2u_3', \\
\{u_1,u_3'\}
    &= \{ bu_2+cu_3, u_1\} 
    = \{ cu_3,u_1\} = c(bu_1u_2+cu_1u_3)
    = cu_1(bu_2 + cu_3) = cu_1u_3'.
\end{align*}
That is, $P$ has skew-symmetric structure as in Example \ref{ex.skewsym}.
\end{proof}

In subsequent parts of this section, we will examine the Poisson algebras defined in Theorem \ref{thm.taft} and extensions of these actions. First, however, we show that the Taft action on these algebras does not extend to the universal enveloping algebra.

\begin{proposition}\label{prop.taft_ext}
Let $P=\kk[u_1,\hdots,u_m]$ be a Poisson algebra. Suppose $T=T_n(\lambda)$ acts linearly on $P$ and the action is given as in \eqref{eq.taft}. Then the action of $T$ does not extend to an action on $U(P)$ according to \eqref{eq.hopf_ext}.
\end{proposition}
\begin{proof}
By Theorem \ref{thm.enveloping}, it suffices to check \eqref{eq.hopf_ext_cond1} and \eqref{eq.hopf_ext_cond2} for the action of $g$ and $x$ on the generators. However, since $g$ is grouplike, these conditions are satisfied trivially. 

We show that the action of $x$ fails \eqref{eq.hopf_ext_cond1}. In the notation of Theorem \ref{thm.enveloping}, let $a=b=u_1$, then \eqref{eq.hopf_ext_cond1} says
\begin{align*}
(\mu_{g \cdot u_1}\nu_{x \cdot u_1} + \mu_{x \cdot u_1}\nu_{u_1})
    - (\mu_{x \cdot u_1}\nu_{g \cdot u_1} + \mu_{u_1}\nu_{x \cdot u_1})
    &= (\lambda\inv \mu_{u_1}\nu_{u_2} + \mu_{u_2}\nu_{u_1})
    - (\lambda\inv\mu_{u_2}\nu_{u_1} + \mu_{u_1}\nu_{u_2}) \\
    &= (\lambda\inv -1)(\mu_{u_1}\nu_{u_2} - \mu_{u_2}\nu_{u_1}) \neq 0.
\end{align*}
This gives the result.
\end{proof}

\subsection{The linear case}

Here we study the Poisson algebra $P=\kk[u_1,u_2,u_3]$ with bracket \eqref{eq.taft_linear}. We consider the rigidity of the Taft action presented above, as well as its extensions.

Observe first that \eqref{eq.taft_linear} is the Kostant-Krillov bracket corresponding to the 3-dimensional Lie algebra $\fg$ with bracket $[u_1,u_2]=0$, $[u_2,u_3]=u_2$, $[u_1,u_3]=u_1$. It may also be realized as the Poisson Ore extension $\kk[u_1,u_2][u_3,\delta]$ where $\delta(u_1)=u_1$ and $\delta(u_2)=u_2$. That is, $\delta$ is the \emph{Euler derivation} given, in general, by $\delta(f)=|f| f$ for homogeneous $f \in \kk[u_1,u_2]$.

The Rees algebra of $P$ is $R(P)=\kk[u_1,u_2,u_3,t]$ with bracket
\[ 
    \{u_i,t\} = \{ u_1,u_2\} = 0, \qquad 
    \{u_3,u_2\} = u_2t, \qquad 
    \{u_3,u_1\} = u_1t.
\]
By Theorem \ref{thm.rees}, the action of $T$ on $P$ extends to an action on $R(A)$ by setting $g(t)=t$ and $x(t)=0$.

It is easy to see that $P$ is not unimodular. In particular, the modular derivation of $P$ applied to $u_3$ is
\[
\phi(u_3) = \sum_{j=1}^3 \frac{\partial}{\partial u_j} \{u_3,u_j\}
    = \frac{\partial}{\partial u_1} \{u_3,u_1\} + \frac{\partial}{\partial u_2} \{u_3,u_2\} = -2 \neq 0.
\]
Set $v_1=u_1^n$, $v_2 = u_2$, and $v_3 = u_3$. Then $P^T = \kk[v_1,v_2,v_3]$ with bracket 
\[ 
    \{ v_1,v_2 \} = 0, \qquad 
    \{ v_2, v_3 \} = v_2, \qquad 
    \{ v_1,v_3 \} = n v_1.
\]
A similar computation shows that $P^T$ is also not unimodular.

We wish to show that $P \niso P^T$, and we do this by comparing Poisson automorphism groups.

\begin{lemma}\label{lem.taft_linear1}
Let $P=\kk[u_1,u_2,u_3]$ be a Poisson algebra with bracket \eqref{eq.taft_linear}. The Poisson automorphism group of $P$ is $\PAut(P) \iso \GL_2(\kk) \rtimes G$ where $G$ is the set of triangular automorphisms given by
    \[  u_1 \mapsto u_1 \qquad 
    u_2 \mapsto u_2 \qquad 
    u_3 \mapsto u_3 + f, \quad f \in \kk[u_1,u_2]. 
    \]
\end{lemma}
\begin{proof}
Suppose $w \in P$ is Poisson normal and write
$w = \sum_{i=0}^d f_i u_3^i$ with $f_i \in \kk[u_1,u_2]$, $f_d \neq 0$. Then $w$ induces a Poisson derivation $\alpha$ such that
\[ w\alpha(u_1) = \{w,u_1\} = \sum_{i=0}^d f_i \{u_3^i,u_1\}
    = \sum_{i=0}^d i f_i u_1 u_3^{i-1}.
\]
Comparing degrees of $u_3$ implies that $d=0$. Hence, the set of Poisson normal elements of $P$ is a subset of $\kk[u_1,u_2]$.

Now write $w = w_0 + w_1 + \cdots \in \kk[u_1,u_2]$ where $w_i$ has degree $i$. Then $\{ w, u_3 \} = \sum i w_i$ so $w=w_i$ for some $i$. Hence, if $\psi \in \PAut(P)$, then $\psi(u_1), \psi(u_2) \in \kk u_1 + \kk u_2$. Write $\psi(u_3) = \sum_{i=0}^d h_i u_3^i$ with $h_i \in \kk[u_1,u_2]$. Then 
\begin{align}\label{eq.linear_aut} 
\kk u_1 + \kk u_2 \ni \psi(u_2) = \{ \psi(u_2),\psi(u_3) \}
    = \sum_{i=0}^d h_i \{ \psi(u_2),u_3^i \}
    = \sum_{i=0}^d h_i i u_3^{i-1} \psi(u_2).
\end{align}
It follows that $d=1$. That is, $\psi(u_3)=\alpha u_3 + f$ for $\alpha \in \kk^\times$ and $f \in \kk[u_1,u_2]$. A straightforward computation using the bracket shows that $\alpha=1$.
\end{proof}

\begin{proposition}\label{prop.taft_linear1}
Let $P=\kk[u_1,u_2,u_3]$ be a Poisson algebra with bracket \eqref{eq.taft_linear}. Suppose $T=T_n(\lambda)$ acts on $P$ such that the action is given as in \eqref{eq.taft}. Then $P^T \niso P$.
\end{proposition}
\begin{proof}
We claim that $\PAut(P) \niso \PAut(P^T)$. In light of Lemma \ref{lem.taft_linear1}, it suffices to compute $\PAut(P^T)$.
A similar computation to that lemma shows that the Poisson normal elements form a subset of $\kk[v_1,v_2]$. Note that, for $P^T$, $u_3$ still induces the Euler derivation when we set $\deg(v_1)=n>1$. Hence, if $\psi \in \PAut(P^T)$, then $\psi(u_2)=a u_2$, $a \in \kk^\times$, and $\psi(u_1)=bu_1+ cu_2^n$ for $b,c \in \kk$ with $b \neq 0$. The computation \eqref{eq.linear_aut} then gives that $\psi(u_3)=u_3 + f$ for $\alpha \in \kk^\times$ and $f \in \kk[u_1,u_2]$. The result follows.
\end{proof}

\subsection{The quadratic case}
Next we investigate the bracket in \eqref{eq.quad_taft}. We denote this family of Poisson algebras by
$P_c = \kk[u_1,u_2,u_3]$.

\begin{proposition}
\begin{enumerate}
    \item The Poisson algebra $P_c$ is unimodular if and only if $c=0$.
    \item The Poisson algebras $P_c$ and $P_{c'}$ are isomorphic if and only if $c'=\pm c$.
    \item Suppose $T=T_n(\lambda)$ acts on $P_c$ such that the action is given as in \eqref{eq.taft}. Then $P_c^T \niso P_c$.
\end{enumerate}
\end{proposition}
\begin{proof}
(1) This is well known (see, e.g., \cite[Theorem 3.1]{CGWZ4}). It is also easy to compute the modular derivation directly. We see that $\phi(u_1)=c u_1$, $\phi(u_2)=c u_2$, and 
\[
\phi(u_3) = \sum_{j=1}^3 \frac{\partial}{\partial u_j} \{u_3,u_j\}
    = \frac{\partial}{\partial u_1} \{u_3,u_1\} + \frac{\partial}{\partial u_2} \{u_3,u_2\} = -2c u_3.
\]

(2) This follows almost immediately from \cite[Theorem 4.6]{GXW}.

(3) Set $v_1=u_1^n$, $v_2 = u_2$, and $v_3 = u_3$. Then $P_c^T = \kk[v_1,v_2,v_3]$ with bracket 
\begin{align}\label{eq.Ac_rels}
    \{ v_1,v_2 \} = 0, \qquad 
    \{ v_2, v_3 \} = cv_2v_3, \qquad 
    \{ v_1,v_3 \} = n c v_1v_3.
\end{align}
This is again a quadratic Poisson algebra with skew-symmetric structure, and so again the result follows from \cite[Theorem 4.6]{GXW}.
\end{proof}

Taft actions on quantum affine spaces were studied in \cite{CG} under the assumption that the parameters are all roots of unity of order $\geq 3$. This is to ensure that the grouplike element $g$ acts diagonally. Using quantization, we obtain a Taft action on a non-PI quantum 3-space.

\begin{example}\label{ex.quant1}
Let $P_c$ be the Poisson algebra with bracket as in \eqref{eq.quad_taft}. Here we have $c_{13}^{13}=c_{23}^{23}=c$ and all others zero. Setting $\hbar=1$, the quantization is $\kk\langle y_1,y_2,y_3\rangle$ modulo the following relations:
\[
[y_1,y_2] =
[y_2,y_3] - \frac{c}{2}(y_2y_3 + y_3y_2) = 
[y_1,y_3] - \frac{c}{2}(y_3y_1 + y_1y_3) = 0.
\]
Assume $c\neq 2$. Set $q=\frac{2+c}{2-c}$. 
Then the above relations simplify to
\[
y_1y_2 - y_2y_1 = 0, \qquad
y_2y_3 - qy_3y_2 = 0, \qquad
y_1y_3 - qy_3y_1 = 0.
\]
This is a quantum 3-space. 

The Taft action \eqref{eq.taft} on $P_c$ 
extends to an action on the quantization by setting:
\begin{align*}
    g(y_1)&=\lambda\inv y_1 & g(y_2)&=y_2 & g(y_3)&=y_3 \\
    x(y_1)&= y_2 & x(y_2)&=0 & x(y_3)&=0.
\end{align*}

Even though the Poisson center satisfies $P_c$ is $\kk$ (when $c \neq 0$), this does not violate \cite[Theorem 3.3]{EW3} because the Poisson center of $Q(P_c)$ is nontrivial. For example, $u_1u_2\inv$ is Poisson central in $Q(P_c)$.
\end{example}

\section{Questions}
In this final section, we pose some questions related to the results presented in this paper.

Most of our results regarding Hopf actions on quadratic Poisson algebras require that the Hopf actions preserve the natural gradings or filtrations on the underlying polynomial algebras. For example, without these conditions, additional Taft algebra actions are present, as the next example illustrates.

\begin{example}[{\cite[Example 3.6]{EW3}}]
Consider the Poisson algebra $P=\kk[u_1,u_2,u_3]$ with bracket 
\[ \{u_1,u_2\} = u_1u_2, \quad \{u_1,u_3\} = u_1u_3, \quad \{u_3,u_2\} = u_2u_3.\] 
Define an action of the Sweedler algebra $T=T_2(-1)$ on $P$ by
\begin{align*}
    g \cdot u_1 &= u_1 &
    g \cdot u_2 &= u_2 &
    g \cdot u_3 &= -u_3 \\
    x \cdot u_1 &= 0    &
    x \cdot u_2 &= 0    &
    x \cdot u_3 &= u_1u_2.
\end{align*}
By the cited reference, $T$ acts on the associative algebra $P$. It is clear that $g$ acts as a Poisson automorphism of $P$. There are only two nontrivial checks for the $x$-action:
\begin{align*}
x \cdot (\{u_1,u_3\}-u_1u_3)
    &= \{g(u_1),x(u_3)\} - g(u_1)x(u_3)
    = \{ u_1,u_1u_2\} - u_1(u_1u_2)
    = u_1(u_1u_2) - u_1(u_1u_2) = 0 \\
x \cdot (\{u_3,u_2\}-u_2u_3)
    &= \{x(u_3),x(u_2)\} - g(u_2)x(u_3)
    = \{ u_1u_2,u_2\} - u_2(u_1u_2)
    = (u_1v)u_2 - v(u_1u_2) = 0.
\end{align*}
\end{example}

We wonder if the rigidity theorem still holds if we relax the linearity of the Hopf actions on the underlying polynomial algebras. 

\begin{question}
Suppose a Taft algebra $T_n(\lambda)$ acts on a quadratic Poisson algebras $\kk[x_1,\ldots,x_m]$. Under what conditions do these actions factor through a group algebra? 
\end{question}

In \cite{CEW1,CEW2}, it is proved that on the associative algebra side, if a finite-dimensional Hopf algebra acts on the $n$th Weyl algebra, the action must factor through a group algebra. In Corollary \ref{cor.weyl}, we demonstrated that the Poisson version of the above rigidity theorem holds for Weyl Poisson algebras, provided that the Hopf actions preserve the natural filtration on the Weyl Poisson algebra. This applies to the idea of Rees algebras associated with filtered Poisson algebras. We wonder if the following conjecture holds in its full generality. 

\begin{conjecture}\label{conj.pweyl}
If $H$ is a finite-dimensional Hopf algebra that acts on a Weyl Poisson algebra, then the $H$-action must factor through a group algebra.
\end{conjecture}

Similarly, our Corollary \ref{cor.Gskew} is a Poisson version of \cite[Theorem 1.9]{EW2016} where the finite Hopf actions on skew-symmetric Poisson algebras are assumed to preserve the gradings of their underlying polynomial algebra. We wonder if it is possible to relax this assumption.

\begin{question}
Let $\bc=(c_{ij})\in M_n(\kk)$ be a skew-symmetric matrix and $\PC$ as in Example \ref{ex.skewsym}. What are the necessary conditions on $\bc$ for any finite Hopf action on $\PC$ to factor through a group action?
\end{question}

Our Proposition \ref{prop.taft_ext} shows that there are generally no nontrivial extensions of a Taft action on a Poisson polynomial algebra to its universal enveloping algebra. On the other hand, Corollary \ref{cor.cocomm} suggests that any cocommutative Hopf action on a Poisson polynomial algebra does extend to its universal enveloping algebra. We propose a question regarding the actions of other noncommutative Hopf algebras. 

\begin{question}
Is there any family of finite-dimensional Hopf algebras that are not cocommutative, whose actions on a Poisson polynomial algebra can be naturally extended to its universal enveloping algebra?
\end{question}

\end{document}